\documentclass{amsart}

\usepackage[all]{xy}
\usepackage{amsmath,amssymb,amsthm}
\usepackage{graphicx,epsfig}
\usepackage{enumerate}
\usepackage{tikz}
\usepackage{ytableau}
\usetikzlibrary{arrows,shapes,positioning}
\usetikzlibrary{decorations.markings}
\usepackage{caption}

\numberwithin{equation}{section}

\swapnumbers
\theoremstyle{plain}
\newtheorem{theorem}{Theorem}[section]

\newtheorem{corollary}[theorem]{Corollary}

\newtheorem{proposition}[theorem]{Proposition}

\theoremstyle{definition}
\newtheorem{example}[theorem]{Example}

\newtheorem{definition}[theorem]{Definition}

\newcommand\scalemath[2]{\scalebox{#1}{\mbox{\ensuremath{\displaystyle #2}}}}

\DeclareMathOperator{\pd}{pd}
\DeclareMathOperator{\reg}{reg}

\DeclareMathOperator{\V}{Vert}

\tikzstyle arrowstyle=[scale=1.5]
\tikzstyle directed=[postaction={decorate,decoration={markings,
    mark=at position .65 with {\arrow[arrowstyle]{latex}}}}]

     \title{Minimal Free Resolutions of $2\times n$ Domino Tilings}

     \author{Rachelle R. Bouchat}
     \address{Department of Mathematics, Indiana University of Pennsylvania, 1011 South Drive, Indiana, PA, 15705, USA}
     \email{ rbouchat@iup.edu}

     \author{Tricia Muldoon Brown}
     \address{Department of Mathematics, Armstrong State University, 11935 Abercorn Street, Savannah, GA 31419, USA}
     \email{patricia.brown@armstrong.edu}


\begin{document}

\begin{abstract}
We introduce a squarefree monomial ideal associated to the set of domino tilings of a $2\times n$ rectangle and proceed to study the associated minimal free resolution.  In this paper, we use results of Dalili and Kummini to show that the Betti numbers of the ideal are independent of the underlying characteristic of the field, and apply a natural splitting to explicitly determine the projective dimension and Castelnuovo-Mumford regularity of the ideal.
\end{abstract}
\maketitle

\section{Introduction}
Squarefree monomial ideals lie in the intersection of three areas of mathematics: commutative algebra, combinatorics, and simplicial topology.  As initial ideals of any arbitrary ideal in a polynomial ring are monomial, much effort has been made to understand monomial ideals, with the groundwork on the study of initial ideals due to Stanley and Reisner.  (See Miller and Sturmfels~\cite{Miller_Sturmfels} for a comprehensive overview of these results.)  Many natural combinatorial objects can be used to generate squarefree monomial ideals.  Making a connection to the field of graph theory, Conca and De Negri~\cite{Conca_DeNegri} introduced the study of edge ideals which are squarefree monomial ideals generated from the edges of a graph.  Of interest is to explicitly describe the minimal free resolution associated to the monomial ideal using combinatorial objects.  In particular, we wish to associate the Betti numbers of the minimal free resolution with these combinatorial or topological objects.  Edge ideal results have been extended to the study of path ideals by Bouchat, H\'a, and O'Keefe in~\cite{Bouchat_Ha_Okeefe} and further generalized to facet ideals of simplicial complexes by Faridi~\cite{Faridi} and to path ideals of hypergraphs by H\'a and Van Tuyl~\cite{Ha_VanTuyl}. 

For the purposes of this paper, we associate squarefree monomial ideals with domino tilings of $2\times n$ rectangles, that is disjoint arrangements of  $2\times 1$ tiles placed horizontally or vertically to completely cover the area of the rectangle.  Domino tilings are well-studied classical combinatorial objects with many interesting properties.  For example, a classic exercise can show the number of domino tilings of a $2\times n$ rectangle is given by the $(n-1)^{st}$ Fibonacci number, and a survey paper by Ardilla and Stanley~\cite{Ardila_Stanley} gives many results for domino and more generalized tilings of the plane.  Tilings have been studied in terms of their enumeration, intersection, and graph theory.  See Fisher and Temperley~\cite{Fisher_Temperley} and Kasteleyn~\cite{Kasteleyn} for the first enumerative results, Butler, Horn, and Tressler~\cite{Butler} for intersection results, or Benedetto and Loehr~\cite{Benedetto_Loehr} for graph theoretical results.  This myriad of results suggests that interpretation as monomial ideals will also be of interest.  In this work, we show the Betti numbers of the ideal corresponding to the set of all $2\times n$ domino tilings are independent of the characteristic of the underlying field.  Further, we apply a natural splitting to the facet ideals generated by the ideals of domino tilings to give a recursion on the Betti numbers and determine the projective dimension and Castelnuovo-Mumford regularity.

We note, the generating elements of these ideals formed from the set of all $2 \times n$ domino tilings can also be viewed as paths of a graph.  However, the ideals are not path ideals, as their generating sets do not correspond to all paths of a specified length within a graph, but rather just a subset (see Example~\ref{pathIdealsEx}).

\section{Minimal free resolutions and domino tilings}

%
We begin with some background on minimal free resolutions. Let $R=k[x_1,\ldots,x_n]$ be the polynomial ring in the variables $x_1,\ldots, x_n$ over the field $k$, and let $M$ be a finitely generated graded $R$-module.  Associated to $M$ is a \emph{minimal free resolution}, which is of the form
	\[0\rightarrow\bigoplus_{\bf a} R(-{\bf a})^{\beta_{p,{\bf{a}}}(M)}\stackrel{\delta_p}{\longrightarrow}\bigoplus_{\bf a} R(-{\bf a})^{\beta_{p-1,{\bf a}}(M)}\stackrel{\delta_{p-1}}{\longrightarrow}
	  \cdots\stackrel{\delta_1}{\longrightarrow}\bigoplus_{\bf a} R(-{\bf a})^{\beta_{0,{\bf a}}(M)}\rightarrow M\rightarrow 0\]
where the maps $\delta_i$ are exact and where $R(-{\bf a})$ denotes the translation of $R$ obtained by shifting the degree of elements of $R$ by ${\bf a}\in\mathbb{N}^n$.  The numbers $\beta_{i,{\bf a}}(M)$ are called the \emph{multi-graded Betti numbers} (or \emph{$\mathbb{N}^n$-graded Betti numbers}) of $M$, and they correspond to the number of minimal generators of degree ${\bf a}$ occurring in the $i^{th}$-syzygy module of $M$.  Of more interest here are the \emph{graded Betti numbers} (or \emph{$\mathbb{N}$-graded Betti numbers}) of $M$ defined as $\displaystyle{\beta_{i,j}(M):=\bigoplus_{a_1+\cdots +a_n=j}\beta_{i,{\bf a}}(M)}$ where ${\bf a}=(a_1,\ldots, a_n)$.  

There are two invariants corresponding to the minimal free resolution of $M$ which measure the size of the resolution.
\begin{definition} Let $M$ be a finitely generated graded $R$-module.  
\begin{enumerate}
	\item The \emph{projective dimension} of $M$, denoted $\pd(M)$, is the length of the minimal free resolution associated to $M$.
	\item The \emph{Castelnuovo-Mumford regularity} (or \emph{regularity}), denoted $\reg(M)$, is 
		\[\reg(M):=\max\{j-i \mid \beta_{i,j}(M)\neq 0\}.\]
\end{enumerate}
\end{definition}

We will define the squarefree monomial ideals of interest in this paper by associating dominos from a tiling of a $2\times n$ array as follows:  
\begin{enumerate}
\item[i.] From left to right, each vertical position composed of squares $(1,k)$ and $(2,k)$ is associated with the domino identified by the variable $y_k$ for $1\leq k \leq n$.  
\item[ii.] Each horizontal position composed of squares $(1,k)$ and $(1,k+1)$ is associated with the domino given by the variable $x_k$ for $1\leq k \leq n-1$; and similarly, each horizontal position with squares $(2,k)$ and $(2,k+1)$ is labeled with the domino variable $x_{n-1+k}$ for $1\leq k \leq n-1$.  
\end{enumerate}
As the dominos are disjoint, a tiling of the array by $n$ dominos can then be represented as a squarefree monomial of degree $n$.  More formally, we state the following definitions.

\begin{definition} Consider a $2\times n$ rectangle, $D$.
\begin{enumerate}
	\item A \emph{tiling} $\tau$ of $D$ is a degree $n$ squarefree monomial $\tau=z_1 z_2 \cdots z_n$ where \\
	$z_i \in \{x_1, x_2,  \ldots, x_{2n-2}, y_1, y_2, \ldots, y_n\}$ and, when considering the variables as dominos in the array, we have $z_i \cap z_j = \emptyset$ for all $1\leq i, j \leq n$.
	We will let $T_n$ denote the set of all tilings $\tau$ of $D$.
	
	\item The \emph{domino ideal} corresponding to $T_n$ is the ideal generated by all tilings in $T_n$, i.e. 
		\[I_n := (\tau \mid \tau\in T_n) \subseteq R=k[x_1, \ldots, x_{2n-2}, y_1, \ldots, y_n].\]
\end{enumerate}
\end{definition}
We note, because these are tilings of rectangles with height two, if $x_k | \tau$ then consequently $x_{n-1+k} | \tau$.  Further, in a slight abuse of notation, we will let $\tau$ denote both the tiling and the monomial in the ring $R$ corresponding to the tiling.

\begin{example}\label{pathIdealsEx}
Consider the domino ideal, $I_3$, corresponding to the domino tilings of a $2\times 3$ rectangle:
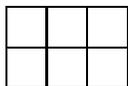
\begin{center}
\begin{tabular}{lclcl}
  \begin{ytableau}
    *(white)&*(white) &*(white) \\
   *(white) &  *(white) &*(white) 
  \end{ytableau}

& \begin{minipage}[h]{.15in}\end{minipage} &
\begin{minipage}[h]{2in}$T_n=\{x_1x_3y_3,x_2x_4y_1,y_1y_2y_3\}$\end{minipage}
& \begin{minipage}[h]{.15in}\end{minipage} &
\begin{minipage}[h]{2in}$I_n = (x_1x_3y_3,x_2x_4y_1,y_1y_2y_3)$\end{minipage}
\end{tabular}
\end{center}
Notice that the generators of $I_3$, namely the tilings of the $2\times 3$ rectangle, are a subset of the paths of length two in the path graph of length six given below:
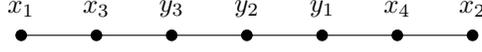
\begin{center}
\begin{tikzpicture}
	\draw (0,0)--(6,0);
	\draw[fill=black] (0,0) circle (2pt);
	\node at (0,.35) {$x_1$};
	\draw[fill=black] (1,0) circle (2pt);
	\node at (1,.35) {$x_3$};
	\draw[fill=black] (2,0) circle (2pt);
	\node at (2,.35) {$y_3$};
	\draw[fill=black] (3,0) circle (2pt);
	\node at (3,.35) {$y_2$};
	\draw[fill=black] (4,0) circle (2pt);
	\node at (4,.35) {$y_1$};
	\draw[fill=black] (5,0) circle (2pt);
	\node at (5,.35) {$x_4$};
	\draw[fill=black] (6,0) circle (2pt);
	\node at (6,.35) {$x_2$};
\end{tikzpicture}
\end{center}
\end{example}

In order to understand the minimal free resolutions of the ideals $I_n$ for $n\geq 1$, we need to investigate how the generating domino tilings interact.  In particular, we wish to understand the graded Betti numbers, $\beta_{i,j} (I_n)$, in terms of sets of domino tilings of the $2\times n$ array; that is, in terms of sets $S=\{\tau_1, \tau_2, \ldots, \tau_k\}$ where $\tau_i \in T_n$ for all $1\leq i \leq k$ and for some $k\geq 0$.

Each set of domino tilings may be represented by a squarefree monomial $\mathbf{x}_S$ with variables from the set $\{x_1, x_2, \ldots, x_{2n-2}, y_1, y_2, \ldots, y_n\}$ where $a | \mathbf{x}_S$ if and only if $a | \tau_i$ for some $\tau_i \in S$.   We note, that the set of domino tilings $S$ associated with a squarefree monomial $\mathbf{x}_S$ is not unique.  Thus, sets of domino tilings are first grouped into equivalence classes by distinct monomials with $\mathbf{x}$ being the representative of the equivalence class.  Example~\ref{ex_eqclass} illustrates such an equivalency.
\begin{example}\label{ex_eqclass}  The pair of tilings 
\[S_1 = \left\{  
  \begin{ytableau}
   *(yellow) x_1 & *(yellow) &*(yellow) x_3 &*(yellow) & *(green) y_5\\
   *(yellow)  x_5&  *(yellow) &*(yellow) x_7&*(yellow)& *(green) 
  \end{ytableau}
,  \begin{ytableau}
   *(green) y_1&*(green)y_2&*(green)y_3 &*(yellow)x_4 &*(yellow)  \\
      *(green)&*(green)&*(green) &*(yellow)x_8 &*(yellow) 
  \end{ytableau}
  \right\} \]
is equivalent to the pair of tilings 
\[ S_2 = \left\{
  \begin{ytableau}
   *(yellow) x_1 & *(yellow) & *(green) y_3 &*(yellow) x_4 &*(yellow) \\
   *(yellow) x_5 &  *(yellow)& *(green) &*(yellow) x_8 &*(yellow) 
  \end{ytableau}
,   \begin{ytableau}
   *(green) y_1&*(green) y_2 &*(yellow) x_3 &*(yellow) &*(green) y_5 \\
      *(green)&*(green) &*(yellow)x_7 &*(yellow) &*(green)
  \end{ytableau} 
  \right\}, \]
 because the corresponding monomial $\mathbf{x}=x_1 x_3 x_4 x_5 x_7 x_8 y_1 y_2 y_3 y_5$ is the same for both pairs of tilings.
\end{example}

As before, without loss of generality, the monomial $\mathbf{x}$ may represent both a monomial in the ring $R$ as well as the set of domino tilings $\mathbf{x} =\{ \tau\in T_n\hspace{1mm}:\hspace{1mm} \tau | \mathbf{x}\}$.

In addition to these definitions, we need the following results from simplicial topology.

\begin{definition}\hspace{1mm}\\
\vspace{-.3in}
\begin{enumerate}
	\item An abstract \emph{simplicial complex}, $\Delta$, on a vertex set $\mathcal{X}=\{x_1,\ldots,x_n\}$ is a collection of subsets of $\mathcal{X}$ satisfying:
	\begin{enumerate}
		\item $\{x_i\}\in\Delta$ for all $i$, and
		\item $F\in\Delta$, $G\subset F\Longrightarrow G\in\Delta$.
	\end{enumerate}
	The elements of $\Delta$ are called \emph{faces} of $\Delta$, and the maximal faces (under inclusion) are called \emph{facets} of $\Delta$.  The simplicial complex $\Delta$ with facets $F_1,\ldots, F_s$ will be denoted by $\langle F_1,\ldots, F_s\rangle$.
	\item For any $\mathcal{Y}\subseteq\mathcal{X}$, an \emph{induced subcollection} of $\Delta$ on $\mathcal{Y}$, denoted by $\Delta_{\mathcal{Y}}$, is the simplicial complex whose vertex set is a subset of $\mathcal{Y}$ and whose facet set is given by $\{F\hspace{1mm}|\hspace{1mm}F\subseteq\mathcal{Y}\mbox{ and }F\mbox{ is a facet of }\Delta\}$.
	\item If $F$ is a face of $\Delta=\langle F_1,\ldots, F_s\rangle$, the \emph{complement} of $F$ in $\Delta$ is given by $F_{\mathcal{X}}^c=\mathcal{X}\setminus F$, and the \emph{complementary complex} is then $\Delta_{\mathcal{X}}^c=\langle (F_1)_{\mathcal{X}}^c,\ldots, (F_s)_{\mathcal{X}}^c\rangle$.
\end{enumerate}
\end{definition}

In a extension of a result of Hochster, this complementary complex, $\Delta_X^c$, can be used to determine the Betti numbers of the ideal.

\begin{theorem}[Alilooee and Faridi, Theorem 2.8 in~\cite{Faridi_Alilooee}]\label{Hochster}
Let $R=K[x_1,\ldots, x_n]$ be a polynomial ring over a field $K$, and $I$ be a pure squarefree monomial ideal in $R$.  Then the $\mathbb{N}$-graded Betti numbers of $I$ are given by
	\[\displaystyle{\beta_{i,d}(I)=\sum_{\Gamma\subset\Delta (I), \mid\V{\Gamma}\mid=d}\dim_K\widetilde{H}_{i-1}(\Gamma_{\V(\Gamma)}^c)}\]
where the sum is taken over the induced subcollections of $\Gamma$ of $\Delta(I)$ which have $d$ vertices.
\end{theorem}

Also of interest, are two subcomplexes of $\Delta$.
\begin{definition}
Let $\Delta$ be a simplicial complex having vertex set $V=\{x_1,\ldots, x_n\}$, and let $x \in V$.
\begin{enumerate}
	\item The \emph{deletion} of $\Delta$ with respect to the vertex $x$ is $del_\Delta (x) = \Delta|_{V\setminus \{x\}}$.
	\item The \emph{link} of $x$ in $\Delta$ is $lk_\Delta (x) = \{F\in \Delta | F\cap \{x\} = \emptyset \hbox{ and } F \cup \{x\} \in \Delta\}$.
\end{enumerate}
\end{definition}

The following example illustrates the link and deletion complexes for the domino tilings of a $2\times 4$ rectangle.

\begin{example}\label{ex_n=4}
In the case $n=4$, the ideal generated by the set of all possible domino tilings is
\begin{equation*}
I_4 =(x_1 x_3 x_4 x_6, x_1x_4y_3y_4, x_2x_5y_1y_4, x_3x_6y_1y_2, y_1y_2y_3y_4),
\end{equation*}
with complementary complex given by
\begin{equation*}
\Gamma_4^c = \langle x_2x_5y_1y_2y_3y_4, x_2x_3x_5x_6y_1y_2, x_1x_3x_4x_6y_2y_3, x_1 x_2 x_4 x_5 y_3 y_4, x_1 x_2 x_3 x_4 x_5 x_6 \rangle.
\end{equation*}
The deletion and link with respect to the vertex $y_3$ are:
\begin{eqnarray*}
del(y_3) &=& \langle x_2x_5y_1y_2 y_4, x_2x_3x_5x_6y_1y_2, x_1x_3x_4x_6y_2, x_1 x_2 x_4 x_5 y_4, x_1 x_2 x_3 x_4 x_5 x_6 \rangle\\
lk(y_3) &=&  \langle x_2x_5y_1y_2y_4,  x_1x_3x_4x_6y_2, x_1 x_2 x_4 x_5  y_4 \rangle
\end{eqnarray*}
Let us identify the vertices corresponding to stacked pairs of horizontal tiles; that is, $x_1$ with $x_4$, $x_2$ with $x_5$, and $x_3$ with $x_6$.  Then, the deletion complex of $y_3$ is contractible to Figure~\ref{dely3} on the left, and the link of $y_3$, which is homotopic to a sphere of dimension one, is given by Figure~\ref{lky3} on the right.\\

\begin{center}
\begin{tabular}{lcl}
\begin{minipage}{1.75in}
\begin{tikzpicture}[scale=.5]
\clip(-1,-3.1) rectangle (7,4);
	\draw[fill=red!50] (6,-2)--(4,-1)--(3,1)--(6,-2);
	\draw[fill=black!50] (0,0)--(1.2,3)--(2,-2)--(0,0);
	\draw[fill=blue!50] (1.2,3)--(2,-2)--(4,-1)--(1.2,3);
	\draw[fill=blue!50] (1.2,3)--(4,-1)--(3,1)--(1.2,3);
	\draw[fill=green!70] (4,-1)--(6,-2)--(2,-2)--(4,-1);
	\draw[fill=yellow!80] (0,0) to [out=-130, in=-90](6,-2)--(2,-2)--(0,0);
	\draw (1.2,3)--(2,-2)--(4,-1)--(3,1)--(1.2,3)--(4,-1)--(6,-2);
	\draw(1.2,3)--(0,0)--(2,-2)--(6,-2)--(3,1);
	\draw(0,0) to[out=-130, in=-90] (6,-2);
	\draw[style=dashed] (0,0)--(3,1);
	\draw[style=dashed] (2,-2)--(3,1);	
	\draw[fill=black] (0,0) circle (3pt) node[left]{$y_4$};
	\draw[fill=black] (2,-2) circle (3pt) node[left]{$x_2$};
	\draw[fill=black] (4,-1) circle (3pt) node[below]{$x_3$};
	\draw[fill=black] (6,-2) circle (3pt) node[right]{$x_1$};
	\draw[fill=black] (3,1) circle (3pt) node[right]{$y_2$};
	\draw[fill=black] (1.2,3) circle (3pt) node[above]{$y_1$};
\end{tikzpicture}
\captionof{figure}{$del(y_3)$}
\label{dely3}
\end{minipage}
& \hspace{.25in} &
\begin{minipage}{1.75in}
\begin{tikzpicture}[scale=.5]
\clip(-1,-3.1) rectangle (7,4);
	\draw[fill=red!50] (6,-2)--(4,-1)--(3,1)--(6,-2);
	\draw[fill=black!50] (0,0)--(1.2,3)--(2,-2)--(0,0);
	\draw[fill=black!50] (1.2,3)--(2,-2)--(3,1);
	\draw[fill=yellow!80] (0,0) to [out=-130, in=-90](6,-2)--(2,-2)--(0,0);
	\draw (1.2,3)--(2,-2);
	\draw (4,-1)--(3,1)--(1.2,3);
	\draw (4,-1)--(6,-2);
	\draw(1.2,3)--(0,0)--(2,-2)--(6,-2)--(3,1);
	\draw(0,0) to[out=-130, in=-90] (6,-2);
	\draw[style=dashed] (0,0)--(3,1);
	\draw[style=dashed] (2,-2)--(3,1);	
	\draw[fill=black] (0,0) circle (3pt) node[left]{$y_4$};
	\draw[fill=black] (2,-2) circle (3pt) node[left]{$x_2$};
	\draw[fill=black] (4,-1) circle (3pt) node[below]{$x_3$};
	\draw[fill=black] (6,-2) circle (3pt) node[right]{$x_1$};
	\draw[fill=black] (3,1) circle (3pt) node[right]{$y_2$};
	\draw[fill=black] (1.2,3) circle (3pt) node[above]{$y_1$};
\end{tikzpicture}
\captionof{figure}{$lk(y_3)$}
\label{lky3}
\end{minipage}
\end{tabular}
\end{center}

In the case $n=3$, the deletion and link with respect to the vertex $y_2$ also correspond to a ball and a sphere in lower dimensions. In this case, the ideal generated by the set of all possible domino tilings is 
\begin{equation*}
I_3=(x_1 x_3 y_3, x_2 x_4 y_1, y_1 y_2 y_3)
\end{equation*}
and the complementary complex is given by 
\begin{equation*}
\Gamma_3^c =\langle x_2 x_4 y_1 y_2, x_1 x_3 y_2 y_3, x_1 x_2 x_3 x_4\rangle.
\end{equation*}
The deletion and link with respect to the vertex $y_2$ are:
\begin{eqnarray*}
del(y_2) &=& \langle x_2 x_4 y_1, x_1 x_3 y_3, x_1 x_2 x_3 x_4 \rangle.\\
lk(y_2) &=& \langle  x_2 x_4 y_1, x_1 x_3 y_3\rangle
\end{eqnarray*}
which, after identifying stacked horizontal dominos, correspond to a 2-dimensional ball and 0-dimensional sphere, respectively.
\end{example}

In the next section, we will utilize the link and deletion complexes and apply results of Dalili and Kummini~\cite{Dalili_Kummini} to show the graded Betti numbers of $I_n$ are independent of the characteristic of the field $k$.


\section{Independence of Characteristic}

In order to show independence of characteristic for domino ideals, we will explore the deletion and link of a certain vertical domino in the complementary complex, but first we need a result on the homotopy type of the complementary complex $\Gamma_n^c$.

\begin{proposition}
\label{gamma_prop}
Let $\Gamma_n^c$ be the complementary complex associated to the domino ideal $I_n$.  Then, $\Gamma_n^c$ is homotopic to a sphere of dimension $n-2$.
\end{proposition}
\begin{proof}
We proceed by induction on $n$.  In the base case, when $n=3$, we have
\begin{equation*} 
\Gamma_3^c =\langle x_2 x_4 y_1 y_2, x_1 x_3 y_2 y_3, x_1 x_2 x_3 x_4\rangle.
\end{equation*}
Applying a deformation retraction which identifies the vertices $x_1$ and $x_2$, respectively, with $x_3$ and $x_4$, respectively, $\Gamma_3^c$ is homotopic to the complex $\langle x_2 y_1 y_2, x_1 y_2 y_3, x_1 x_2 \rangle$, that is, the complex of two triangles adjoined at a point with two other vertices connected by a line.  Thus $\Gamma_3^c$ is homotopic to a 1-dimensional sphere.

For $n>3$, we separate the generators of $I_n$ into sets of those containing $y_n$ and those containing $x_{n-1}$ and, continuing to identify $x_k$ with $x_{k+n-1}$, the complementary complex is the union
\begin{equation*}
\Gamma_n^c = \Gamma_{n-1}^c * x_{n-1} \cup \Gamma_{n-2} *x_{n-2} y_{n-1} y_n.
\end{equation*}
Because $\Gamma_{n-1}^c$ can similarly be written as $\Gamma_{n-2}^c * x_{n-2} \cup \Gamma_{n-3}^c*y_{n-1}$ and because every face in $\Gamma_{n-3}^c$ is a face of $\Gamma_{n-2}^c$, we see every face of $\Gamma_{n-1}^c$ is contained in a face of $\Gamma_{n-2}^c*x_{n-2} y_{n-1}$.  Further, each face $F \in \Gamma_{n-1}^c$ is contained in the face $F*y_n \in \Gamma_{n-2}^c * x_{n-2} y_{n-1}y_n$.  In particular, because the faces of $\Gamma_{n-1}^c$ make up a $(n-3)$-dimensional sphere and because $y_n$ is not contained in $\Gamma_{n-1}^c$, we have that $\Gamma_{n-1}^c \cup \Gamma_{n-2}^c * x_{n-2}y_{n-1}y_n$ is the cone by $y_n$ over the sphere and thus is homotopic to a $(n-2)$-dimensional ball.  Without loss of generality, we may assume the faces of  $\Gamma_{n-1}^c$ are the boundary of this $(n-2)$-dimensional ball.  Finally, we note that $x_{n-1}$ is not contained in any face of $\Gamma_{n-2}^c*x_{n-2} y_{n-1} y_n$, and so we see that $\Gamma_{n-1}^c * x_{n-1} \cup \Gamma_{n-2} *x_{n-2} y_{n-1} y_n$ is the cone by $x_{n-1}$ of the boundary of the $(n-2)$-dimensional ball. Therefore, $\Gamma_n^c$ is homotopic to a sphere of dimension $n-2$.
\end{proof}

Now consider the link and deletion of the complementary complex with respect to the vertical domino~$y_{n-1}$.

\begin{proposition}
\label{link_deletion_prop}
Let $\Gamma_n^c$ be the complementary complex associated to the domino tiling ideal $I_n$.  Then, the link, $lk_{\Gamma_n^c} (y_{n-1})$, is homotopic to a sphere of dimension $n-3$ for $n\geq 3$.  Further, the deletion complex, $del_{\Gamma_n^c}(y_{n-1})$, is contractible.
\end{proposition}

\begin{proof}
Divide the set of $2\times n$ tilings into three disjoint sets as follows:
\begin{enumerate}[i.]
\item $A_n=\{ \tau \in T_{n} : y_{n-1} y_{n} | \tau\}$
\item $B_n=\{\tau \in T_{n} : x_{n-1} x_{2n-2} | \tau\}$
\item $C_n = \{\tau \in T_{n} : x_{n-2} x_{2n-3} y_n | \tau\}$ 
\end{enumerate}

\noindent Each of these sets is described by the tiles comprising the rightmost two or three dominos of the tiling.  The facet complexes can also be described this way, because the complements will not contain these two or three dominos.  The facet complexes generated by the complements of the elements in these sets are
\begin{enumerate}[i.]
\item $A_n^c = \Gamma_{n-2}^c * \{x_{n-1} x_{n-2} x_{2n-3} x_{2n-2}\}$
\item $B_n^c =\Gamma_{n-2}^c * \{x_{n-2} x_{2n-3} y_{n-1} y_{n}\}$
\item $C_n^c = \Gamma_{n-3}^c *\{ x_{n-3} x_{n-1} x_{2n-4} x_{2n-2} y_{n-1} y_{n-2}\}$
\end{enumerate}
where $*$ represents the simplicial join.  We then have
\begin{equation*}
del (y_{n-1} ) = \Gamma_{n}^c \setminus \{y_{n-1}\} = A_n^c \cup (B_n^c \setminus \{y_{n-1}\}) \cup (C_n^c \setminus \{y_{n-1}\})
\end{equation*}
and 
\begin{equation*}
lk (y_{n-1} ) = (B_n^c \setminus \{y_{n-1}\}) \cup (C_n^c \setminus \{y_{n-1}\}).
\end{equation*}
Using a straight-forward deformation retraction, we simplify the complexes by identifying the vertices $x_{2n-4}$, $x_{2n-3}$, and $x_{2n-2}$, respectively, with their corresponding horizontal tiles $x_{n-3}$, $x_{n-2}$, and $x_{n-1}$, respectively.  Thus,
\begin{equation*}
lk (y_{n-1}) = (\Gamma_{n-2}^c * \{x_{n-2} y_n\}) \cup (\Gamma_{n-3}^c * \{x_{n-3} x_{n-1} y_{n-2}\}).
\end{equation*}

Applying Proposition~\ref{gamma_prop}, $\Gamma_{n-2}^c$ is homotopic to a sphere of dimension $n-4$.  By taking the cone over $\Gamma_{n-2}^c$ by the line $\langle x_{n-2} y_n\rangle$ we remove the singularity and effectively fill in the $(n-4)$-dimensional hole, creating a complex homotopic to a ball of dimension $n-3$.

We note that the vertices $x_{n-2}$, $x_{n-1}$, and $y_n$ are not elements of either $\Gamma_{n-2}^c$ nor $\Gamma_{n-3}^c$.  So without loss of generality, we can assume $\Gamma_{n-2}^c$ is on the boundary of the $(n-3)$-dimensional ball $\Gamma_{n-2}^c * \{x_{n-2} y_n\}$.  Further, both $x_{n-3}$ and $y_{n-2}$ are included in $\Gamma_{n-2}^c$. In fact, the maximal faces of $\Gamma_{n-2}^c$ can be partitioned disjointly into faces which contain $x_{n-3}$ and faces which contain $y_{n-2}$, because the generating tilings must have exactly one of these two tiles.  More specifically, 

\begin{equation}\label{recurrence_eq}
\Gamma_{n-2}^c = (\Gamma_{n-3}^c *\{x_{n-3}\} ) \cup (\Gamma_{n-4}^c *\{x_{n-4} y_{n-3} y_{n-2}\}).
\end{equation}

We can now consider the rest of the complex, namely $\Gamma_{n-3}^c *\{x_{n-3} x_{n-1} y_{n-2}\}$.  We know that $\Gamma_{n-3}^c$ is homotopic to a $(n-5)$-dimensional sphere and is a subcomplex of $\Gamma_{n-2}^c$ by extension of the relationship given above in Equation~\ref{recurrence_eq}.  We wish to show $\Gamma_{n-3}^c *\{x_{n-3} y_{n-2}\}$ contains every facet of $\Gamma_{n-2}^c$; therefore by adjoining $x_{n-1}$, we create a cone over the boundary of a $(n-3)$-dimensional ball which is a homotopic to a sphere dimension $n-3$.

First, we easily have that $\Gamma_{n-3}^c *\{x_{n-3} y_{n-2}\}$ contains all faces of $\Gamma_{n-2}^c$ which contain $x_{n-3}$ as needed for the first term in the union given in Equation~\ref{recurrence_eq}.  Next, we note that $\Gamma_{n-4}^c *\{x_{n-4}\}$ is contained in the complex $\Gamma_{n-3}^c$, so $\Gamma_{n-3}^c *\{y_{n-2}\}$ contains all facets $F*\{x_{n-4} y_{n-2}\}$ where $F\in \Gamma_{n-4}^c$.  Further, every maximal face on the boundary $\Gamma_{n-2}^c$ that contains $x_{n-4} y_{n-2}$ also contains the vertex $y_{n-3}$.  As $y_{n-3}$ is a part of every face in $\Gamma_{n-4}^c *\{x_{n-4} y_{n-3} y_{n-2}\}$, we may simply retract $y_{n-3}$ into any other vertex in the part of the boundary determined by $\Gamma_{n-4}^c *\{x_{n-4} y_{n-3} y_{n-2}\}$. Thus we have homotopy equivalence, and we see that $C_n^c \setminus \{y_{n-1}\}$ is the cone of the vertex $x_{n-1}$ over the complex $\Gamma_{n-3}^c$ which is homotopic to the boundary of a $(n-3)$-dimensional ball given by $B_n^c \setminus \{y_{n-1}\}$.  As the interior of this ball, given by faces containing $x_{n-2}$ or $y_n$, is not contained in $C_n^c \setminus \{y_{n-1}\}$, therefore, $lk (y_{n-1})$ is homotopic to the sphere $S^{n-3}$.

It is now straight-forward to show the $del(y_{n-1})$ is contractible.  The complex $del (y_{n-1})$ differs from the link by the complex $A_n^c = \Gamma_{n-2}^c * \{x_{n-1} x_{n-2} x_{2n-3} x_{2n-2}\}$.  In particular, the vertex $x_{n-1}$ is joined with $\Gamma_{n-2}^c$.  This join fills the interior of the sphere given by the link.  Thus, $del(y_{n-1})$ is homotopic to a ball of dimension $n-2$, and hence is contractible.  (See Example~\ref{ex_linkdel} for a concrete demostration of this process.)
\end{proof}

The following example illustrates Proposition~\ref{link_deletion_prop} using generating tilings of a $2\times 4$ rectangle.
\begin{example}\label{ex_linkdel}
Recall, the domino ideal corresponding to $T_4$, namely 
\begin{equation*}
I_4 =(x_1 x_3 x_4 x_6, x_1x_4y_3y_4, x_2x_5y_1y_4, x_3x_6y_1y_2, y_1y_2y_3y_4).
\end{equation*}
The generators of $I_4$ are divided into three disjoint sets with corresponding facet complexes:\\

\[\begin{array}{rclcrcl}
A_4 &=& \{x_1x_4y_3y_4, y_1y_2y_3y_4\} & \hspace{.25in} & A_4^c &=&\langle x_2x_3x_4x_6y_1y_2, x_1x_2x_3x_4x_5x_6\rangle\\
&&&&&=&\langle x_2x_3x_5x_6\rangle * \langle y_1y_2, x_1x_4\rangle\\
&&&&&=&\langle x_2x_3x_5x_6\rangle * \Gamma_{2}^c\\
B_4&=&\{x_1x_3x_4x_6, x_3x_6y_1y_2\} & & B_4^c&=& \langle x_2x_5y_1y_2y_3y_4, x_1x_2x_4x_5y_3y_4 \rangle\\
&&&&&=& \langle x_2x_5y_3y_4\rangle * \langle y_1y_2, x_1x_4 \rangle\\
&&&&&=& \langle x_2x_5y_3y_4\rangle * \Gamma_2^c\\
C_4&=&\{x_2x_5y_1y_5\} & & C_4^c&=&\langle x_1x_3x_4x_6y_2y_3\rangle\\
&&&&&=&\langle x_1x_3x_4x_6y_2y_3\rangle*\Gamma_1^c\\
\end{array}\]


\noindent Identifying horizontal pairs, $x_i$ and $x_{n+i-1}$, we have 
\begin{equation*}
lk(y_3)=B_4^c\setminus \{y_3\} \cup C_4^c\setminus \{y_3\} =\langle x_2x_5y_4\rangle * \Gamma_2^c \cup \langle x_1x_3x_4x_6y_2\rangle*\Gamma_1^c. 
\end{equation*}
Thus, $B_4^c\setminus \{y_3\}$ is the tetrahedron $x_2 y_1 y_2 y_4$ along with the triangle $x_1x_2y_4$, and $C_4^c\setminus \{y_3\}$ is the triangle $x_1x_3y_2$.  The tetrahedron $x_2x_3y_1y_2$ and triangle $x_1x_2x_3$ of $A_4^c$, pictured in Figure~\ref{fig_fill_in}, fill in the hole in the link in $del(y_3)=V_4^c \cup B_4^c\setminus \{y_3\} \cup C_4^c\setminus \{y_3\}$.  $A_4^c$ is the simplicial join of the line $x_2 x_3$ with $\Gamma_2^c =\langle x_1, y_1y_2\rangle$.  The complex $del(y_3)$ differs from the link by $A_4^c$, thus $del(y_3)$ is contractible.\\

\begin{center}
\begin{tabular}{lcl}
\begin{minipage}{1.75in}
\begin{tikzpicture}[scale=.5]
\clip(-1,-3.1) rectangle (7,4);
	\draw[fill=blue!50] (1.2,3)--(2,-2)--(4,-1)--(1.2,3);
	\draw[fill=blue!50] (1.2,3)--(4,-1)--(3,1)--(1.2,3);
	\draw[fill=green!50] (4,-1)--(6,-2)--(2,-2)--(4,-1);
	\draw (1.2,3)--(2,-2)--(4,-1)--(3,1)--(1.2,3)--(4,-1)--(6,-2);
	\draw[very thick, red](1.2,3)--(3,1);
	\draw[style=dashed] (2,-2)--(3,1);	
	\draw[fill=black] (2,-2) circle (3pt) node[left]{$x_2$};
	\draw[fill=black] (4,-1) circle (3pt) node[below]{$x_3$};
	\draw[fill=red] (6,-2) circle (3pt) node[right]{$x_1$};
	\draw[fill=black] (3,1) circle (3pt) node[right]{$y_2$};
	\draw[fill=black] (1.2,3) circle (3pt) node[above]{$y_1$};
\end{tikzpicture}
\captionof{figure}{$A_4^c$}
\label{fig_fill_in}
\end{minipage}
& \hspace{.25in} &
\begin{minipage}{1.75in}
\begin{tikzpicture}[scale=.5]
\clip(-1,-3.1) rectangle (7,4);
	\draw[fill=red!50] (6,-2)--(4,-1)--(3,1)--(6,-2);
	\draw[fill=black!50] (0,0)--(1.2,3)--(2,-2)--(0,0);
	\draw[fill=blue!50] (1.2,3)--(2,-2)--(4,-1)--(1.2,3);
	\draw[fill=blue!50] (1.2,3)--(4,-1)--(3,1)--(1.2,3);
	\draw[fill=green!70] (4,-1)--(6,-2)--(2,-2)--(4,-1);
	\draw[fill=yellow!80] (0,0) to [out=-130, in=-90](6,-2)--(2,-2)--(0,0);
	\draw (1.2,3)--(2,-2)--(4,-1)--(3,1)--(1.2,3)--(4,-1)--(6,-2);
	\draw(1.2,3)--(0,0)--(2,-2)--(6,-2)--(3,1);
	\draw(0,0) to[out=-130, in=-90] (6,-2);
	\draw[style=dashed] (0,0)--(3,1);
	\draw[style=dashed] (2,-2)--(3,1);	
	\draw[fill=black] (0,0) circle (3pt) node[left]{$y_4$};
	\draw[fill=black] (2,-2) circle (3pt) node[left]{$x_2$};
	\draw[fill=black] (4,-1) circle (3pt) node[below]{$x_3$};
	\draw[fill=black] (6,-2) circle (3pt) node[right]{$x_1$};
	\draw[fill=black] (3,1) circle (3pt) node[right]{$y_2$};
	\draw[fill=black] (1.2,3) circle (3pt) node[above]{$y_1$};
\end{tikzpicture}
\captionof{figure}{$del(y_3)$}
\end{minipage}
\end{tabular}
\end{center}
\end{example}

We now wish to apply the results of Dalili and Kummini~\cite{Dalili_Kummini} to show independence of characteristic.  
\begin{proposition}[Dalili and Kummini, Remark 2.2 and Discussion 2.5]\label{prop_DK1}  Let $\Delta$ be a simplicial complex.
\begin{enumerate}
	\item Let $V$ be the vertex set of $\Delta$ with $x\in V$.  If $\tilde{H}_* (del_{\Delta} (x); \mathbb{Z}) =0$, then $\tilde{H}_{i+1} 
	(\Delta; \mathbb{Z}) \simeq \tilde{H}_i (lk_{\Delta} (x); \mathbb{Z})$ for all $i\geq 0$.
	
	\item Let $I$ be the Stanley Reisner ideal of $\Delta$.  Then, $\beta (I)$ depends on char $\mathbb{K}$ if and only if the groups $H_* (\Delta; \mathbb{Z})$ have torsion.
\end{enumerate}
\end{proposition}

Thus, we can state our main theorem.
\begin{theorem}\label{thm_independent}
Let $R=k[x_1,\ldots,x_{2n-2},y_1,\ldots,y_n]$, and let $I_n$ be the domino ideal corresponding to $T_n$.  Then $\beta_{i,j}(I_n)$ is independent of $\mbox{char}(k)$ for all $i,j\in\mathbb{Z}$.
\end{theorem}
\begin{proof}
The theorem is a consequence of Proposition~\ref{link_deletion_prop} and Dalili and Kummini's remarks, because $\tilde{H}_{i-1} (\Gamma_{n}^c; \mathbb{Z}) \simeq \tilde{H}_i (lk_{\Gamma_{n}^c}(y_{n-1}); \mathbb{Z})$ and as spheres the homology groups $H_* (lk_{\Gamma_{n}^c}(y_{n-1}); \mathbb{Z})$ have no torsion.
\end{proof}

In the next section, we describe a splitting of the domino ideal, $I_n$, into two smaller ideals, $V_n$ and $U_n$, in order to provide a recursive description for the graded Betti numbers of $I_n$ as well as determine projective dimension and Castelnuovo-Mumford regularity.


\section{Splitting Ideals}

Splittable monomial ideals were introduced by Eliahou and Kervaire~\cite{Eliahou_Kervaire}  in order to separate a monomial ideal into simpler components.  We give the definition from their work where $G(I)$ denotes the canonical generating set of the ideal $I$.
\begin{definition}
We say that $I$ is \textit{splittable} if $I$ is the sum of two non-zero monomial ideals $V$ and $U$ such that
\begin{enumerate}
\item $G(I)$ is the disjoint union of $G(V)$ and $G(U)$, and
\item there is a \textit{splitting function}
\begin{eqnarray*}
G(V\cap U) &\rightarrow& G(V) \times G(U)\\
w & \rightarrow& (\phi(w), \psi(w))\\
\end{eqnarray*}
satisfying the following properties:
\begin{enumerate}
\item[(S1)] $w=\mbox{lcm}(\phi(w),\psi(w))$ for all $w\in G(W) =G(V\cap U)$, and
\item[(S2)] for every subset $G'\subset G(W)$, both $\mbox{lcm}$ $\phi(G')$ and $\mbox{lcm}$ $\psi(G')$ strictly divide $\mbox{lcm}$ $G'$.
\end{enumerate}
\end{enumerate}
\end{definition}

We will split the generators of the ideal $I_n$ into disjoint sets by rightmost dominoes in the tiling.  A tiling must either end with the vertical domino, $y_n$, or two horizontal dominoes, $x_{n-1}$ and $x_{2n-2}$.  Thus, set $V_n$ to be the ideal generated by the set of tilings containing a rightmost vertical tile; that is, $V_n=\{\tau \in T_n\hspace{1mm}:\hspace{1mm} y_n | \tau\}$.  Let $U_n$ be the ideal generated by tilings containing two rightmost horizontal tiles; that is, $U_n=\{\tau \in T_n \hspace{1mm}:\hspace{1mm} x_{n-1}x_{2n-2} | \tau\}$.  We check the conditions for this to be a splitting.

\begin{proposition}\label{firstsplit}
Let $T_n$ be the set of domino tilings of a $2\times n$ rectangle, and let $I_n$ be the domino ideal associated to $T_n$.  Then $I_n=V_n+U_n$ is a splitting (where $V_n$ and $U_n$ are defined as above).
\end{proposition}
\begin{proof}

By definition, the set of minimal monomial generators of $I_n$, $G(I_n)$, is the disjoint union of the sets of minimal monomial generators of $V_n$ and $U_n$, $G(V_n)$ and $G(U_n)$ (respectively) as a domino tiling can either contain the tile $y_n$ or the pair of tiles $\{x_{n-1}, x_{2n-2}\}$ but not both.  Further, we define the splitting function $(\phi, \psi):\hspace{1mm} G(V_n \cup U_n) \rightarrow G(V_n) \times G(U_n)$ where $\displaystyle{\phi (w) = \frac{w}{x_{n-1} x_{2n-2}}}$ and $\displaystyle{\psi(w)= \frac{w}{y_n}}$.  Now, we see 
\begin{equation*}
\mbox{lcm} (\phi(w), \psi(w)) = \mbox{lcm}\left(\frac{w}{x_{n-1} x_{2n-2}}, \frac{w}{y_n}\right) = w
\end{equation*}
so (S1) is satisfied.  To check condition (S2), we have

\[\scalemath{.93}{\mbox{lcm} (\phi(S)) = \displaystyle{\mbox{lcm} \left(\frac{w_1}{x_{n-1} x_{2n-2}},\frac{w_2}{x_{n-1}x_{2n-2}}, \ldots, \frac{w_k}{x_{n-1} x_{2n-2}}\right)} =\displaystyle{\frac{\mbox{lcm}(w_1, w_2, \ldots, w_k)}{x_{n-1}x_{2n-2}}}< \mbox{lcm}(w_1, w_2, \ldots, w_k) = \mbox{lcm} (S)},\]


\noindent and similarly for $\psi(S)$.  Thus, $I_n$ can be split by $V_n$ and $U_n$.
\end{proof}

We can now apply the following theorem due to Eliahou and Kervaire~\cite{Eliahou_Kervaire} for Betti numbers, which appears as a condition to Francisco, H\'a, and Van Tuyl's Proposition 2.1 categorizing Betti splittings in~\cite{Francisco_Ha_VanTuyl}.
 

\begin{theorem}\label{splitting_theorem}
Suppose that $I$ is a splittable monomial ideal with splitting $I=V+U$.  Then for all $i,j\geq 0$,
\begin{equation*}
\beta_{i,j} (I) = \beta_{i,j}(V) + \beta_{i,j} (U) + \beta_{i-1,j}(V\cap U).
\end{equation*}
\end{theorem}


First we explore the ideals $V_n$, $U_n$, and $V_n \cap U_n$. 


\begin{proposition}\label{horizontal_vertical_prop}
The graded Betti numbers of $V_n$ and $U_n$ are given by the graded Betti numbers of $I_{n-1}$ and $I_{n-2}$, respectively.  That is, 
\[\begin{array}{rcl}
\beta_{i,j} (V_n) & = &  \beta_{i,j-1}(I_{n-1})\\
\beta_{i,j} (U_n) & = & \beta_{i,j-2}(I_{n-2})
\end{array}\]
for $i,j \geq 0$. 
\end{proposition}

\begin{proof}
The results follow immediately after observing that $V_n = y_n I_{n-1}$ and $U_n = x_{n-1}x_{2n-2} I_{n-2}$.
\end{proof}

We also have another splittable ideal, $V_n\cap U_n$.  These two splittable ideals, namely $I_n$ and $V_n\cap U_n$, will play an integral role in the proof of the formula for $\pd(I_n)$ in Proposition~\ref{pdProp}.


\begin{proposition}\label{SplitIntersection}
The intersection $V_n \cap U_n$ can be split into $V_{n} \cap U_n = \widehat{V}_n + \widehat{U}_n$ where
\[\begin{array}{rcl}
	\widehat{V}_n & = & \{\tau \in V_n\cap U_n \hspace{1mm}:\hspace{1mm} y_{n-1} | \tau\}\\
	\widehat{U}_n & = &\{\tau \in V_n\cap U_n \hspace{1mm}:\hspace{1mm} x_{n-2} x_{2n-3} | \tau\}.
\end{array}\]
\end{proposition}
\begin{proof}
We are splitting an ideal's minimal generators into two disjoint sets indexed by a vertical domino and by the corresponding pair of horizontal dominos.  The splitting function is defined similarly to the  splitting $I_n=V_n + U_n$ given in Proposition~\ref{firstsplit}, so we omit it here.
\end{proof}

Then it follows by the result of Eliahou and Kervaire~\cite{Eliahou_Kervaire} that
\begin{equation*}
\beta_{i-1,j}(V_{n} \cap U_n) = \beta_{i-1,j}(\widehat{V}_n) + \beta_{i-1,j}(\widehat{U}_n) + \beta_{i-2, j} (\widehat{V}_n \cap \widehat{U}_n),
\end{equation*}

\begin{proposition}\label{prop_relationsVnUn}
Using the previous definitions of $V_n$, $U_n$, $\widehat{V}_n$, and $\widehat{U}_n$, we have:
\begin{enumerate}
\item $\beta_{i-1,j} (\widehat{V}_n ) = \beta_{i-1,j-4}(I_{n-2})$,
\item $\beta_{i-1,j} (\widehat{U}_n) = \beta_{i-1,j-2}(V_{n-1} \cap U_{n-1})$, and
\item $\beta_{i-2,j} (\widehat{V}_n \cap \widehat{U}_n) = \beta_{i-2,j-3} (V_{n-1} \cap U_{n-1})$.
\end{enumerate}
\end{proposition}

\begin{proof}
Notice that every minimal generator in the ideal $V_n\cap U_n$ contains the factors $y_n$, $x_{n-1}$, and $x_{2n-2}$.  Let $\widehat{V}_n$and $\widehat{U}_n$ be defined as in Proposition~\ref{SplitIntersection}.

To check the remaining claims, we observe that all tilings of a $2\times (n-2)$ array from the set $T_{n-2}$ correspond to a minimal generator in both $V_n$ and $U_n$, namely $y_{n-1} y_n T_{n-2} \subseteq V_n$ and $x_{n-1} x_{2n-2} T_{n-2} \subseteq U_n$. Thus, minimal generators in $\widehat{V}_n$ correspond to elements of the form $y_{n-1} y_n x_{n-1} x_{2n-2} \tau$ for $\tau \in T_{n-2}$, and hence $\beta_{i-1,j} (\widehat{V}_n ) = \beta_{i-1,j-4}(I_{n-2})$.

Next, we note $y_{n-1} \nmid \tau \in \widehat{U}_n$, so we consider the map $\Phi : \widehat{U}_n \longrightarrow V_{n-1} \cap U_{n-1}$ where if $\tau = \widehat{\tau}x_{n-2} x_{n-1} x_{2n-3} x_{2n-2} y_n$, then $\Phi(\tau) = \widehat{\tau} x_{n-2} x_{2n-3} y_{n-1}$.  Clearly $\Phi(\tau) \in V_{n-1} \cap U_{n-1}$, as the map only removes two dominos, $x_{n-1}$ and $x_{2n-2}$, and relabels another domino.  The map may be reversed giving a bijection, and thus $\beta_{i-1,j} (\widehat{U}_n) = \beta_{i-1,j-2}(V_{n-1} \cap U_{n-1})$.

Similarly, define the bijection $\phi: \widehat{V}_n \cap \widehat{U}_n \longrightarrow V_{n-1} \cap U_{n-1}$ for which the element \\ $\tau = \widehat{\tau} x_{n-2} x_{n-1} x_{2n-3} x_{2n-2} y_{n-1}y_n\in \widehat{V}_n\cap\widehat{U}_n$ is mapped onto the element $\widehat{\tau}x_{n-2} x_{2n-3} y_{n-1}\in V_{n-1}\cap U_{n-1}$ by removing the three dominos $x_{n-1} x_{2n-2} y_n$.  This proves $\beta_{i-2,j} (\widehat{V}_n \cap \widehat{U}_n) = \beta_{i-2,j-3} (V_{n-1} \cap U_{n-1})$.
\end{proof}


\begin{proposition}\label{prop_recurrence}
Let $T_n$ be the set of domino tilings of a $2\times n$ rectangle, and let $I_n$ be the domino ideal corresponding to $T_n$.  Then for $n\geq 4$,
\begin{eqnarray*}
\beta_{i,j} (I_n) &=& \beta_{i,j-1}(I_{n-1})+\beta_{i,j-2}(I_{n-2}) + \sum_{m=0}^{n-4} \sum_{k=0}^m {m\choose k} \beta_{i-1-k,j-4-2m-k}(I_{n-2-m}) \\
&&+ \sum_{k=0}^{n-4} {n-4\choose k} \left( \beta_{i-1-k, j-2n+6-k} (V_3 \cap U_3) + \beta_{i-2-k, j-2n+5-k}(V_3 \cap U_3)\right)
\end{eqnarray*}
where $\beta_{1,5}(V_3\cap U_3) = \beta_{1,6}(V_3\cap U_3) =\beta_{2,7} (V_3 \cap U_3) =1$ and $\beta_{i,j} (V_3\cap U_3) =0$ elsewhere.
\end{proposition}

\begin{proof}
Applying Theorem~\ref{splitting_theorem} and Proposition~\ref{horizontal_vertical_prop}, we have
		\[\begin{array}{rcl}
			\beta_{i,j}(I_n) &=& \beta_{i,j}(V_n)+\beta_{i,j}(U_n)+\beta_{i-1,j}(V_n\cap U_n)\\
						& = & \beta_{i,j-1}(I_{n-1})+\beta_{i,j-2}(I_{n-2})+\beta_{i-1,j}(V_n\cap U_n)\\
		\end{array}\]

It is immediate that the first two summands agree.  We now apply Propositions~\ref{SplitIntersection} and \ref{prop_relationsVnUn}, to repeatedly split the ideal $(V_{k} \cap U_{k})$ for $k=n-1, n-2, \ldots, 4$.  Each time the ideal is split, we have
\begin{equation*}
\beta_{i',j'}(V_{n'} \cap U_{n'})= \beta_{i',j'-4} (I_{n'-2}) + \beta_{i', j'-2}(V_{n'-1}\cap U_{n'-1}) + \beta_{i'-1, j'-3}(V_{n'-1}\cap U_{n'-1}). 
\end{equation*}
In particular, consider the first splitting of $V_n\cap U_n$:
\begin{equation}\label{eq_firstsplit}
\beta_{i-1,j}(I_n)=\beta_{i-1,j-2}(V_{n-1}\cap U_{n-1}) +\beta_{i-2,j-3}(V_{n-1}\cap U_{n-1})
\end{equation}
The index $m$ in the sum of Proposition~\ref{prop_recurrence} counts the number of times the Propositions~\ref{SplitIntersection} and \ref{prop_relationsVnUn} have been applied to equation~\ref{eq_firstsplit}; that is, the number of times we have split the ideal $V_{k}\cap U_k$ in order to calculate Betti numbers in terms of $I_{k-2}$ and $V_{k-1}\cap U_{k-1}$.  For $m>0$, each $\beta_{i', j'}(I_{n-2-m})$ appears as the summand of a splitting applied to $\beta_{i',j'+4}(V_{n-m}\cap U_{n-m})$.  To determine how many times this term appears in the sum we note, the recursive relationship between the Betti numbers with index $i'$ and $n'$ is the same as that of binomial coefficients; thus any time the term $
\beta_{i',j'}(V_{n'} \cap U_{n'})$ appears we find, through splitting, the terms $ \beta_{i', j'-4} (V_{n'-1} \cap U_{n'-1}) + \beta_{i'-1, j-3}(V_{n'-1}\cap U_{n'-1})$.
Therefore a particular Betti number was determined by a higher-dimensional Betti number such that either the $i'$ index remains constant and the $j'$ index decreases by two or the $i'$ index decreases by one and the $j'$ decreases by three.  The index $k$ in Proposition~\ref{prop_recurrence} counts the number of times in $m$ splittings that the index $i'$ decreases by one.  Thus, beginning with the index $i-1$ there are ${m\choose k}$ terms who index has decreased to $i-1-k$.  Further, each splitting decreases the second index $j$ by two or three, and precisely $k$ of those times it decreases by three.  Hence, the $j$ index decreases to $j-2m-k$; that is presence of the term $\beta_{i-1,j} (V_n \cap U_n)$ implies ${m\choose k}$ copies of $\beta_{i-1-k, j-2m-k} (V_{n-m} \cap U_{n-m})$ which in turn implies ${m\choose k}$ copies of $\beta_{i-1-k, j-2m-k-4}(I_{n-m-2})$.

Now applying this to Equation~\ref{eq_firstsplit} and summing over each $0\leq k \leq m$ and $m\geq 0$, it is left to check the end conditions.  When $m=n-4$, Propositions~\ref{SplitIntersection} and \ref{prop_relationsVnUn} are applied to the intersection ideal $V_{n-m} \cap U_{n-m} = V_{4} \cap U_4$.  As there are ${n-4\choose k}$ copies of the Betti number $\beta_{i-1-k,j-2n-k+8}(V_4\cap U_4)$ for $0\leq k \leq n-4$, when this ideal is split and for all $k$, we are left with $\scalemath{.95}{\sum_{k=0}^{n-4} {n-4\choose k} \left( \beta_{i-1-k, j-2n+6-k} (V_3 \cap U_3) + \beta_{i-2-k, j-2n+5-k}(V_3 \cap U_3)\right)}$, completing the sum.
\end{proof}

\begin{proposition}\label{pdProp}
Let $T_n$ be the set of domino tilings of a $2\times n$ rectangle, and let $I_n$ be the domino ideal corresponding to $T_n$.  Then, $\pd(I_n)=n-1$.
\end{proposition}
\begin{proof}	
	We will induct on $n\geq 3$.  For $n=3$, $I_3=(y_1y_2y_3,x_1x_3y_3,x_2x_4y_1)$ and $\pd(I_n)=2$ with $\beta_{2,7}(I_3)=1$.  Assume $\pd(I_{n-1})=n-2$ and is achieved at $\beta_{n-2, 3n-5}(I_{n-1}) =1$.  Since $I_n=V_n+U_n$ is a splitting for all $n\geq 3$ Theorem~\ref{splitting_theorem} and Proposition~\ref{horizontal_vertical_prop} provide:
		\[\begin{array}{rcl}
			\beta_{i,j}(I_n) & = & \beta_{i,j}(V_n)+\beta_{i,j}(U_n)+\beta_{i-1,j}(V_n\cap U_n)\\
			                       & = & \beta_{i,j-1}(I_{n-1})+\beta_{i,j-2}(I_{n-2})+\beta_{i-1,j}(V_n\cap U_n) \tag{$\star$}\\
		\end{array}\]
Furthermore, from Proposition~\ref{SplitIntersection}, we obtain:
		\[\scalemath{.96}{\begin{array}{rcl}
			\beta_{i,j}(I_n) & = & \beta_{i,j-1}(I_{n-1})+\beta_{i,j-2}(I_{n-2})+\beta_{i-1,j}(\widehat{V}_n)+\beta_{i-1,j}(\widehat{U}_n)+\beta_{i-2,j}(\widehat{V}_n\cap\widehat{U}_n)\\
			                       & = & \beta_{i,j-1}(I_{n-1})+\beta_{i,j-2}(I_{n-2})++\beta_{i-1,j-4}(I_{n-2}) +\beta_{i-1,j-2}(V_{n-1}\cap U_{n-1}) +\beta_{i-2,j-3}(V_{n-1}\cap U_{n-1})\\
		\end{array}}\]
Consider $i=n-1$ and $j=3n-2$.  Since $\pd(I_{n-1})=n-2$ and $\pd(I_{n-2})=n-3$, we obtain:
		\[\begin{array}{rcl}
			\beta_{n-1,3n-2}(I_n) & = & \beta_{n-2,3n-4}(V_{n-1}\cap U_{n-1})+\beta_{n-3,3n-5}(V_{n-1}\cap U_{n-1})\tag{$\star\star$}
		\end{array}\]
From the inductive hypothesis, $\pd(I_{n-1})=n-2$ and $\beta_{n-2,3n-5}(I_{n-1}) = 1$.  It follows from ($\star$) that:
		\[\begin{array}{rcl}
			\beta_{n-2,3n-5}(I_{n-1}) & = & \beta_{n-2,3n-6}(I_{n-2})+\beta_{n-2,3n-7}(I_{n-3})+\beta_{n-3,3n-5}(V_{n-1}\cap U_{n-1}).
		\end{array}\]
Since $\pd(I_{n-3})=n-4$ and $\pd(I_{n-2})=n-3$, it follows from the inductive hypothesis that:
		\[\begin{array}{rcl}
			\beta_{n-2,3n-5}(I_{n-1}) & = & \beta_{n-3,3n-5}(V_{n-1}\cap U_{n-1}).
		\end{array}\]
Moreover, since $\pd(I_{n-1})=n-2$, we can conclude that $\beta_{n-3,3n-5}(V_{n-1}\cap U_{n-1})=1$.  Then from $(\star\star)$, $\beta_{n-1,3n-2}(I_n)= 1$.  

Further, because the deletion complex is contractible, Proposition~\ref{prop_DK1} implies the link completely determines the complementary complex $\Gamma_n^c$.  Because the link is homotopic to a sphere of dimension $n-3$, the complementary complex is homotopic to a sphere of dimension $n-2$.  Thus, applying the variant of Hochster's formula due to Allilooee and Faridi in Theorem~\ref{Hochster}, $n-1$ is the maximum index $i$ for non-zero Betti numbers $\beta_{i,j}$, and therefore $\pd(I_n)=n-1$.
\end{proof}

\begin{corollary}
Let $T_n$ be the set of domino tilings of a $2\times n$ rectangle, and let $I_n$ be the domino ideal corresponding to $T_n$.  Then, $\reg(I_n)=2n-1$.
\end{corollary}
\begin{proof}
From Proposition~\ref{pdProp}, $\pd(I_n)=n-1$ with  $\beta_{n-1,3n-2}(I_n)\geq 1$.  Since $3n-2$ is the greatest graded shift, $\reg(I_n)=(3n-2)-(n-1)=2n-1$.
\end{proof}


\end{document}